\documentclass[preprint,12pt]{elsarticle}

\usepackage[latin1]{inputenc}

\usepackage{amssymb}
\usepackage{amsmath}
\usepackage{amsthm}

\newtheorem{theorem}{Theorem}
\newtheorem{lemma}{Lemma}

\newtheorem{proposition}{Proposition}

\usepackage[final]{microtype}

\usepackage{tikz}
\usepackage{subfigure}
\usepackage{caption}
\usepackage{here}

\usepackage{algorithm}
\begin{document}

\begin{frontmatter}

 \author{Steffen Borgwardt}
 \ead{borgwardt@ma.tum.de}
 \address{Technische Universit\"at M\"unchen}

\author{Shmuel Onn}
\ead{onn@ie.technion.ac.il}
 \address{Technion - Israel Institute of Technology, Haifa}

\title{Efficient solutions for weight-balanced partitioning problems}


\begin{abstract}
\noindent We prove polynomial-time solvability of a large class of clustering problems where a weighted set of items has to be partitioned into clusters with respect to some balancing constraints. The data points are weighted with respect to different features and the clusters adhere to given lower and upper bounds on the total weight of their points with respect to each of these features. Further the weight-contribution of a vector to a cluster can depend on the cluster it is assigned to. Our interest in these types of clustering problems is motivated by an application in land consolidation where the ability to perform this kind of balancing is crucial.

Our framework maximizes an objective function that is convex in the summed-up utility of the items in each cluster. Despite hardness of convex maximization and many related problems, for fixed dimension and number of clusters, we are able to show that our clustering model is solvable in time polynomial in the number of items if the weight-balancing restrictions are defined using vectors from a fixed, finite domain. We conclude our discussion with a new, efficient model and algorithm for land consolidation.
\end{abstract}

\begin{keyword}
constrained clustering \sep convex maximization \sep integer programming \sep land consolidation

\MSC[2010] 90C10 \sep 90C25 \sep 90C90 \sep 91C20
\end{keyword}

\end{frontmatter}

\section{Introduction}

Partitioning a set of items while respecting some constraints is a frequent task in exploratory data analysis, arising in both operations research and machine learning; see e.g. \cite{bdw-09, hk-08}. We consider partitioning for which the sizes of the clusters are restricted with respect to multiple criteria. There are many applications where it is necessary to adhere to given bounds on the cluster sizes.

For example, these include the modeling of polycrystals in the material sciences  \cite{abgp-15} and face recognition using meshes, where the original mesh is partitioned into parts of equal sizes to obtain an optimal running time for graph-theoretical methods that are applied to all of these parts \cite{bbfs-11}. Our interest in these types of problems comes from an application in {\em land consolidation}. See for example \cite{bbg-11, b-03, bg-04} and in particular the outreach article \cite{bbg-14} for the impact in academia and practice.

\subsection{Land consolidation}\label{sec:land}

The farmers of many agricultural and private forest regions in Bavaria and Middle Germany own a large number of small lots that are scattered over the whole region. The reasons are strict heritage laws and a frequent change of ownership. There is significant overhead driving and an unnecessarily high cost of cultivation. In such a situation, a land consolidation process may be initiated by the state to improve on the cost-effective structure of the region. Voluntary land exchanges (by means of lend-lease agreements) are a popular method for such a process: The existing lots are kept without changes and the rights of cultivation are redistributed among the farmers of the region.

This corresponds to a combinatorial redistribution of the lots and can be modeled as a clustering problem where each lot is an item and each farmer is a cluster \cite{bg-04}. The main goal is to create large connected pieces of land for each farmer. One way to do so is to represent the lots by their midpoints in the Euclidean plane and to use the geographical locations of the farmsteads of each farmer as a set of sites. Then one performs a weight-balanced least-squares assignment of the lots to these sites \cite{b-10, bbg-11, bg-12}. The result is a redistribution where the farmers' lots lie close to their farmsteads. As a positive sideeffect, many of a farmer's lots are connected and can be cultivated together.

The lots differ in several features like their size, quality of soil, shape, and attached subsidies, and some of these features are even different for each farmer. For example one farmer may be eligible for subsidies if they cultivate a given lot, while another farmer may not. In such a situation, the lot is more valuable for the first farmer.  Here a natural constraint is that - after the redistribution - each farmer should have lots that (approximately) sum up to the farmer's original total with respect to each feature.

Of course, partitioning a weighted set of items into clusters of prescribed sizes ({\em weight-balanced partitioning}) is readily seen to be NP-hard, even for just two clusters and each item having just a single weight that is uniform for both clusters: deciding whether there is such a partition is at least as hard as the {\em Subset Sum} problem. The methods in the literature \cite{b-10, bbg-11, bbg-14, bg-12} approach this intrinsic hardness by solving least-squares problems by an LP relaxation and rounding. The model in \cite{bbg-11} performed particularly well in practice.

In this paper, we will present a general clustering framework that, when applied to land consolidation, improves on this model by dealing with its biggest shortcoming: The model in \cite{bbg-11} (and in fact \cite{b-10, bbg-14, bg-12}, too) is not able to balance the weights of clusters with respect to multiple features of the lots at the same time. Instead the redistribution of cultivation rights is done with respect to a single `value' of a lot, an aggregation of all its properties (and this value is the same for each farmer).

The farmers will not accept a large deviation with respect to any of the features of their total lots. They will only participate in the redistribution if they do not lose a significant area of land, do not lose a lot in quality of soil, and do not lose much of their subsidies at the same time. But the models in the literature will return `optimal' solutions for which not even the aggregated value necessarily is within the specified bounds (e.g. $3\%$) from the original. Further, even if a farmer's aggregated deviation is small, they may have received more land than before, but of much lower quality - which they will not accept. These are intrinsic weaknesses coming from the relaxation and rounding that are performed in \cite{b-10, bbg-11, bbg-14, bg-12}. In the practical implementation, this meant that a lot of the work still had to be done `by hand'.

However, the previous methods do not use all of the favorable properties of the input data for agricultural regions.  The data typically falls into only a fixed number of categories, which just comes from the way the features are measured in practice. For example, especially large slots ($>5$ hectare) are not traded at all. Further, one does not distinguish between lot sizes that differ by less than a tenth of a hectare, so that one obtains a finite domain of lot sizes. The quality of soil is measured with a number between $1$ and $100$, which is a finite domain itself. But in fact in a single agricultural region it is rare to have more than five different values within this range. The same happens for the subsidies attached to lots and other measures. With this additional assumption, we are able to present an algorithm that solves the problem of land consolidation exactly and efficiently.


\subsection{Our contributions}

In view of the above application, in the present paper we consider a new, generalized class of clustering problems where weight-balancing restrictions are defined using vectors from a fixed, finite domain. Each item specifies a vector of weights for each cluster.  These weight vectors represent the weight the item would contribute with respect to the different features if assigned to the respective cluster. The sizes of the clusters are bounded above and below with respect to all features.

Further, each item gives a vector of utility values with respect to each cluster representing the utility gained if the item is assigned to the corresponding cluster. (These vectors do not have to be from a fixed domain.)  Each cluster 'collects' a total utility by summing up the utility of all its items' utility vectors.  We then maximize an objective function that is convex on the sum of utility vectors of each part.

These are the functions used for shaped partitioning \cite{bhr-92, hor-99, o-10, or-04}, which encompass many of the objective functions commonly used in data analysis tasks, such as minimal-variance clustering; see e.g. \cite{or-04}. They are intimately related to the studies of gravity polytopes \cite{b-10, bg-12}. Shaped partitioning is known to be NP-hard if either the dimension of the utility vectors or the numbers of clusters is part of the input \cite{hor-99}, because it captures the hardness of convex maximization due to the possible existence of exponentially many local optima \cite{o-10}. For this, the number of clusters and the dimension of the utility vectors are fixed in our analysis.


Our main result is to show that this framework is polynomial-time solvable in the number of items, provided the weight vectors come from a fixed and finite domain. In Section $2$, we introduce a formal notation for these clustering models, state our main results, and discuss the objective functions that can be represented in our models in some detail. Section $3$ is dedicated to the necessary proofs. They are based on polynomial-size and -time reductions to the maximization of a (special) convex function over a system of constraints with an $n$-fold constraint matrix. The latter can be performed efficiently, which comes from a combination of recent algebraic Graver bases methods \cite{dhk-13, dhorw-09, o-10, o-12} and geometric edge-directions and zonotope methods \cite{hor-99, or-04}. We then transfer the complexity results to a generalized model  for land consolidation, and present a new algorithm for it, in Section $4$.

\section{Model and results}\label{sec:model}
In the following, we partition $n$ items $\{1,\dots,n\}$ into $p$ clusters $\pi_1,\dots,\pi_p$. Each item $j$ has a utility matrix $C^j\in \mathbb{R}^{d\times p}$. Its $i$-th column $C_i^j$ represents a vector of $d$ utility values gained if item $j$ is assigned to cluster $\pi_i$.

The utility of a clustering $\pi=(\pi_1,\dots,\pi_p)$ is $f(\sum\limits_{j\in \pi_1} C_1^j,\dots,\sum\limits_{j\in \pi_p} C_p^j)$, where $f:\mathbb{R}^{d\times p} \rightarrow \mathbb{R}$ is a convex function presented by a comparison oracle. Note that $\sum\limits_{j\in \pi_i} C_i^j$ is the sum of utility vectors of the items in cluster $\pi_i$.

Recall that a comparison oracle for $f$ is a routine that when queried on two values $y, z$ returns whether $f(y)\leq f(z)$ or not. In our framework, we do not even need an explicit form of the function, only comparisons of its values. This makes the complexity results presented in the following stronger  - essentially one works with as little information as possible. Of course, for any natural specific convex function that arises in data analysis, such as a norm or squared Euclidean distances (such as the ones that appear in Section $4$), it is easy to explicitly perform these comparisons.

The task is to find a clustering of maximal utility under certain balancing constraints. We begin with a first model for these balancing constraints and extend it only later on. First, assume that we are given a fixed, finite set $\Omega \subset \mathbb{Z}^p$ such that each item $j$ has a weight vector $w^j \in \Omega$. Informally, this vector contains a weight for each cluster $\pi_i$ that it would contribute if assigned to $\pi_i$. Further, each cluster $\pi_i$ defines a total weight $b_i \in \mathbb{Z}$ of the items that have to belong to it. Note that a finite cardinality $m$ of $\Omega$ can e.g. be achieved by bounding the absolute values of all components in the $w^j$: for maximal absolute value $\omega$, we obtain $m\leq (2\omega + 1)^p$.

The corresponding set of restrictions on the cluster sizes can be written as
$$ \sum\limits_{j\in \pi_i} w_i^j = b_i \quad \quad \bigl(i\in [p]\bigr).$$
Let $\Pi$ be the set of all partitions $\pi$ of $\{1,\dots,n\}$ into $p$ clusters. Then a full statement of this optimization problem $(P_1)$ would be
$$
\begin{array}{llcccccl}
\multicolumn{6}{c}
{(P_1) \quad \quad \quad \max\limits_{\pi \in \Pi} \;\  f(\sum\limits_{j\in \pi_1} C_1^j,\dots,\sum\limits_{j\in \pi_p} C_p^j)} &&\\[.1cm]
     \quad \quad \quad           &                           &     &  \sum\limits_{j\in \pi_i} w_i^j       &  =  \quad  b_i  &                      & \qquad &  \bigl(i\in [p]\bigr)
\end{array}
$$
Our first main result is polynomial-time solvability of this model for fixed $d$ and $p$.

\begin{theorem}\label{thm:polysolve}
Suppose $d$ and $p$ are fixed and suppose there is a fixed, finite set $\Omega \subset \mathbb{Z}^p$ such that all $w^j \in \Omega$. Then for every convex $f$ presented by a comparison oracle, problem $(P_1)$ can be solved in polynomial time.
\end{theorem}

Note that $(P_1)$ and Theorem \ref{thm:polysolve} already extend the model and complexity results in Section $3.3$ in \cite{dhorw-09}, where only the number of items in the clusters is balanced. So essentially the weights of all items are $1$ for each cluster. In constrast, we allow for arbitrary weights of the items and these weights may also differ for the different clusters.

In many applications, instead of having exact sizes of the partition parts, one is given lower and upper bounds on the sizes. To extend the above program to lower and upper bounds $b_i^{\pm}\in \mathbb{Z}$ on the total weights of the partition parts, we extend our formulation in several places.
 Formally, these constraints take the form
$$ b_i^-\leq \sum\limits_{j\in \pi_i} w_i^j \leq b_i^+ \quad \quad \bigl(i\in [p]\bigr).$$
For our later proofs, we now rewrite the corresponding optimization problem using only equalities. Let us introduce slack variables $s_i^{\pm}\in \mathbb{Z}$ for all $i\in [p]$. Then the new optimization problem $(P_2)$ can be stated as
$$
\begin{array}{llcccccccl}
\multicolumn{8}{c}
{(P_2) \quad \quad \quad \max\limits_{\pi \in \Pi} \;\  f(\sum\limits_{j\in \pi_1} C_1^j,\dots,\sum\limits_{j\in \pi_p} C_p^j)} &&\\[.1cm]
      \quad \quad \quad&                                    &     &  (\sum\limits_{j\in \pi_i} w_i^j)       & +  & s_i^+ & =  \quad   b_i^+  &                      & \qquad &  \bigl(i\in [p]\bigr)\\
      \quad \quad \quad&            &     &  (\sum\limits_{j\in \pi_i} w_i^j )      &- & s_i^-  & =  \quad  b_i^-  &                      & \qquad &  \bigl(i\in [p]\bigr)\\
\quad \quad \quad& & & &  & s_i^{\pm} & \geq \quad 0\quad &  & \qquad &  \bigl(i\in [p]\bigr)
\end{array}
$$
With some modifications to the construction for $(P_1)$ in the proof of Theorem \ref{thm:polysolve}, we can show a similar statement for this more general class of problems.
\begin{theorem}\label{thm:polysolve2}
Suppose $d$ and $p$ are fixed and suppose there is a fixed, finite set $\Omega \subset \mathbb{Z}^p$ such that all $w^j \in \Omega$. Then for every convex $f$ presented by a comparison oracle, problem $(P_2)$ can be solved in polynomial time.
\end{theorem}

Finally, we extend our model to allow for balanced weights with respect to $s\geq 1$ different features. Instead of item $j$ listing a vector $w^j\in \mathbb{Z}^p$ of weights with respect to the clusters, it now has a matrix $W^j\in \mathbb{Z}^{s\times p}$ listing $s$-dimensional vectors $W_i^{j}$ of weights contributed to the cluster $\pi_i$ it is assigned to. All of these weight matrices come from a fixed, finite set $\Omega \subset \mathbb{Z}^{s\times p}$. As before, a finite cardinality $m$ of $\Omega$ can be achieved by bounding the absolute values of all components in the $W^j$: for maximal absolute value $\omega$, one now obtains $m\leq (2\omega + 1)^{sp}$.

 Further, instead of $b_i^{\pm} \in \mathbb{Z}$, we now use $B_i^{\pm}\in \mathbb{Z}^s$, and likewise we have to use slack vectors $S_i^{\pm}\in \mathbb{Z}^s$ for all $i\in [p]$. This gives us an optimization problem $(P_3)$ as
$$
\begin{array}{llcccccccl}
\multicolumn{8}{c}
{(P_3) \quad \quad \quad \max\limits_{\pi \in \Pi} \;\  f(\sum\limits_{j\in \pi_1} C_1^j,\dots,\sum\limits_{j\in \pi_p} C_p^j)} &&\\[.1cm]
          \quad \quad \quad&                                        &     &  (\sum\limits_{j\in \pi_i} W_i^j)       & +  & S_i^+ & =  \quad   B_i^+  &                      & \qquad &  \bigl(i\in [p]\bigr)\\
      \quad \quad \quad&                    &     &  (\sum\limits_{j\in \pi_i} W_i^j )      &- & S_i^-  & =  \quad  B_i^-  &                      & \qquad &  \bigl(i\in [p]\bigr)\\
    \quad \quad \quad&     & & &  & S_i^{\pm} & \geq \quad 0\quad &  & \qquad &  \bigl(i\in [p]\bigr)
\end{array}
$$
Even for this extension, we keep polynomial-time solvability.
\begin{theorem}\label{thm:polysolve3}
Suppose $d$ and $p$ are fixed and suppose there is a fixed, finite set $\Omega \subset \mathbb{Z}^{s\times p}$ such that all $W^j \in \Omega$. Then for every convex $f$ presented by a comparison oracle, problem $(P_3)$ can be solved in polynomial time.
\end{theorem}

We conclude this section with some general remarks on the expressive power of the objective function. In general, the convex functions $f:\mathbb{R}^{d\times p} \rightarrow \mathbb{R}$ of the form $f(\sum\limits_{j\in \pi_1} C_1^j,\dots,\sum\limits_{j\in \pi_p} C_p^j)$ cover a wide range of objective functions commonly used in clustering. In particular, they directly represent the \emph{aggregation of utility values}. Each of the $p$ clusters $\pi_i$ contributes a summed-up utility vector $\sum\limits_{j\in \pi_i} C_i^j$ of its items and these vectors then are aggregated to a value in $\mathbb{R}$ by the convex function $f$. A simplemost case for $f$ is to just add up the components of all these vectors possibly (scaled by a factor). If the components of all $C_i^j$ are non-negative, such an objective function corresponds to a linear transform of the $l_1$-norm. Recall that all norms are convex functions. By choosing $l_s$ with $s>1$, one values single components in the utility vectors of the clusters higher in relation to many equally large values. For example, for $l_\infty$, only the largest absolute values among all components counts.

By means of our framework, one can represent even more general ways of aggregating the utility by turning to the so-called {\em clustering bodies} \cite{bg-10}: Here one combines two norms, one  $\|\cdot \|$ for $\mathbb{R}^d$ on the parts $\sum\limits_{j\in \pi_i} C_i^j$ and a monotone norm $\|\cdot \|_*$ for $\mathbb{R}^p$ aggregating these values. The objective function then takes the form
 $$f(\sum\limits_{j\in \pi_1} C_1^j,\dots,\sum\limits_{j\in \pi_p} C_p^j)=\|(\|\sum\limits_{j\in \pi_1} C_1^j\|,\dots,\|\sum\limits_{j\in \pi_p} C_p^j\|) \|_*.$$
The level set for value at most one is a convex body, a clustering body, and serves as the unit ball for a semi-norm. This explains convexity of the above $f$.

\section{Proofs}

The proofs for Theorems \ref{thm:polysolve} to \ref{thm:polysolve3} have a common demeanor in that we exhibit polynomial-time and -size transformation of the corresponding model to a variant of shaped partitioning for a set of constraints defined by an $n$-fold matrix.  Generally speaking, we aim for a problem statement of the form
\begin{eqnarray}\max \{ f(Cx): A^{(n)}x=b, l\leq x \leq u, x\in \mathbb{Z}^N\}\end{eqnarray}
with $l,u\in \mathbb{Z}^N$, $C\in \mathbb{Z}^{c\times N}$, and a convex function $f: \mathbb{Z}^c \rightarrow \mathbb{R}$ presented by a comparison oracle. Further, the $n$-fold matrix $A^{(n)}$ is derived from two matrices $A_1$ and $A_2$ by the classical construction in the form
$$A^{(n)} = \left[ {\begin{array}{ccc} A_1 & \dots & A_1 \\ A_2 & &  \\ & \dots & \\  &  & A_2 \end{array} }\right]$$
with $n$ copies of both $A_1$ and $A_2$ arranged in the depicted layout.
Note that $n$ corresponds to the number of items; each of them gets a set of columns with one of the building blocks $A_1$ and $A_2$. Clearly, $N$ has to be a multiple of $n$.

A problem of the form $(1)$ is polynomial-time solvable if $A_1$ and $A_2$ are fixed and $C$ has a constant number of rows $c$. The following proposition sums up this important tool for our proofs.

\begin{proposition} [\cite{dhorw-09}]\label{prop:polytime}
Let $c\in \mathbb{N}$ be fixed, and let $A_1, A_2$ be fixed matrices. Further, let $A^{(n)}$ be the $n$-fold matrix derived from $A_1$ and $A_2$, let $l,u\in \mathbb{Z}^N$, let $C\in \mathbb{Z}^{c\times N}$, and $f:\mathbb{Z}^c \rightarrow \mathbb{R}$ presented by a comparison oracle.

Then there is an algorithm that is polynomial in $n$ and the length of binary input for $C,l,u$ and $b$ that solves
$$\max \{ f(Cx): A^{(n)}x=b, l\leq x \leq u, x\in \mathbb{Z}^N\}.$$
\end{proposition}

In view of Proposition \ref{prop:polytime}, to obtain polynomial-time-solvability for our models, we will show that there is a polynomial-time and -size reduction to the above form in the size of the input for our original models. Note that a representation of the data as above is in stark contrast to the `natural' representation in Section $2$. In addition to the explicit prerequisites of Proposition \ref{prop:polytime}, namely the constant number of rows for $C$ and the fixed $A_1$, $A_2$, we also have to guarantee that $N$ stays polynomial in the size of the original input.

In the following, we carefully perform the necessary reductions. We begin with the first model $(P_1)$ and Theorem \ref{thm:polysolve}, and then subsequently extend the constructions.

\begin{proof}[{\bf \em  Proof of Theorem $1$}]
We prove the claim by transforming $(P_1)$ to a statement in the form
$$\max \{ f(Cx): A^{(n)}x=b', l\leq x \leq u, x\in \mathbb{Z}^N\}.$$
Let us begin with the constraints. Recall that $m$ is the size of $\Omega$. Consider the two matrices
$$ A_1 = \left[ {\begin{array}{ccccccccc} w_1^1 & & &  & & w_1^m & & & \\  & w_2^1 & &  & \dots &  & w_2^m & & \\ & & \dots &  &  \dots &  & & \dots & \\  &  & & w_p^1  & &  &  & & w_p^m \end{array} } \right] \in \mathbb{R}^{p \times mp}$$
and
$$A_2=[ 1\; 1  \dots 1] \in \mathbb{R}^{1 \times mp}.$$

Further, we define the vectors $b=(b_1,\dots,b_p)^T$ and ${\bf{1}}=(1,\dots,1)^T\in \mathbb{R}^n$ to be able to write a system of equations $A^{(n)}x=b'=\left(\begin{array}{c} b \\ {\bf 1} \end{array}\right)$. The variables $x\in \mathbb{R}^{n(mp)}$ here correspond to decision variables in the following way:

We have $x=(x^1,\dots,x^n)^T$, where $x^j$ corresponds to the $j$-th column-block of the $n$-fold matrix, i.e. to the $j$-th 'copy' of $A_1$ and $A_2$. Each $x^j$ takes the form $x^j=(x_1^j,\dots,x_m^j)^T$, i.e. it consists of $m$ blocks $x_i^j\in\mathbb{Z}^{p}$. Note that this construction yields $N=n(mp)$, which is polynomial in the input.

Further, let us define lower and upper bounds on $x$ in the form  $l\leq x \leq u$: Choose lower bounds $l= {\bf 0}\in \mathbb{R}^N$ and define the upper bound vector $u=(u^1,\dots,u^n)^T$ to consist of $n$ blocks $u^j=(u_1^j,\dots,u_m^j)^T\in \mathbb{Z}^{mp}$, setting $u_i^j={\bf 1}\in\mathbb{Z}^p$ if item $j$ has weight vector $w^i$  and  $u_i^j={\bf 0}\in\mathbb{Z}^p$ otherwise.

The system
$$A^{(n)}x=\left(\begin{array}{c} b \\ {\bf 1} \end{array}\right), \quad l\leq x \leq u$$
can be derived in polynomial time and is of polynomial size, as it only uses a polynomial number of copies of numbers from the original input (and of zeroes and ones). Let us discuss why it is equivalent to our original set of constraints. First, note that the lower and upper bounds force an integral solution $x$ to be a $0,1$-solution.

The assignment of item $j$ to a cluster is determined by the decision variables $x^j$ corresponding to the $j$-th block of $A_1$ and $A_2$. The '$A_2$-block' of the system tells us that precisely one entry in $x^j$ is equal to $1$. By the upper bounds on $x$, this can only be the case for an index which corresponds to a correct combination of cluster $\pi_i$ and weight contribution by item $j$. For this, in the  '$A_1$-block' of the system, the correct weight for $x_j$ is added up in the equation to obtain total cluster size $b_i$. Thus all items are assigned and all clusters obtain the correct total weight.

It remains to check whether the objective function can be written in form $f(Cx)$, where $C$ has a constant number of rows. Note that
$$\sum\limits_{j\in \pi_s} C_s^j = \sum\limits_{j=1}^n  \sum\limits_{i=1}^m ((x_i^j)^T {\bf 1}) C_s^j.$$
Thus $C$ takes the form
$$C=[ (C')^1 \dots (C')^1 \dots (C')^n  \dots (C')^n] \in \mathbb{R}^{dp \times n(mp)},$$
with $m$ consecutive copies for each of the $(C')^j$. These are defined as
$$ (C')^j = \left[ {\begin{array}{cccc} C_1^j & & &  \\  & C_2^j & &  \\ & & \dots &  \\  &  & & C_p^j   \end{array} } \right] \in \mathbb{R}^{dp \times p}.$$

The claim now follows from the number of $dp$ rows being constant and observing that $Cx$ yields the vector $((\sum\limits_{j\in \pi_1} C_1^j)^T,\dots,(\sum\limits_{j\in \pi_p} C_p^j))^T)^T$. We satisfy all prerequisites of Proposition \ref{prop:polytime} and are done.
\end{proof}

Next, we prove Theorem \ref{thm:polysolve2} by extending the above construction. The goal is to enter the slack variables $s_i^{\pm}$ into the model while preserving the $n$-fold structure of the program. It is possible to do so by duplicating them and moving them into the building blocks of the matrix as follows.

\begin{proof}[{\bf \em Proof of Theorem $2$}] 

Let $A$, $b$ {etc.} refer to the construction for Theorem \ref{thm:polysolve}. Their corresponding new counterparts are denoted by a bar, e.g. $\bar A$, $\bar b$. We now want a problem statement for $(P_2)$ of the form
$$\max \{ f(\bar Cx): \bar A^{(n)}\bar x=\bar b', \bar l\leq \bar x \leq \bar u, \bar x\in \mathbb{Z}^{\bar N}\},$$
where $\bar A^{(n)}$ is an $n$-fold matrix derived by the standard construction using matrices $\bar A_1$ and $\bar A_2$.

For this formulation, we use a vector $\bar x =(\bar x^1,\dots, \bar x^n)^T$, where the column-blocks $\bar x^j=((s^j)^+,(s^j)^-,x^j)^T$ now also have copies of the slack variables $(s^j)^+=((s_1^j)^+,\dots,(s_p^j)^+)^T$ and $(s^j)^-=((s_1^j)^-,\dots,(s_p^j)^-)^T$ for each of the original $x^j$. Note that $\bar x \in \mathbb{R}^{\bar N}$, where $\bar N=n(2p+mp)$, which again is polynomial in the input.

$\bar A^{(n)}$ is set to
$$\bar A^{(n)} = \left[ {\begin{array}{ccc} \bar A_1 & \dots & \bar A_1 \\ \bar A_2 & &  \\ & \dots & \\  &  & \bar A_2 \end{array} }\right],$$
where $$\bar A_1 = \left[ {\begin{array}{ccc} E & {\bf 0} & A_1 \\  {\bf 0}  & -E & A_1   \end{array} }\right]\in\mathbb{R}^{2p\times (2p+mp)} \text{ and } \bar A_2 =[0\quad  0\quad  A_2]\in \mathbb{R}^{1\times (2p+mp)}$$ with $E\in \mathbb{R}^{p\times p}$ the unit-matrix and ${\bf 0}\in \mathbb{R}^{p\times p}$ the $0$-matrix. Further $\bar b'=(\bar b, \bf{1})^T$ with ${\bf{1}}=(1,\dots,1)^T\in \mathbb{R}^n$ and $\bar b=(b_1^+,\dots,b_p^+,b_1^-,\dots,b_p^-)^T$ listing the upper and lower bounds on the cluster sizes.\footnote{In the notation of $\bar b'$ and in other places, we avoid double transposes when writing vectors for better readability when the context is clear. For example $c=(a, b)^T$ for two column vectors $a,b$ would be a column vector $c$.}

Next, let us define lower and upper bounds $\bar l, \bar u \in \mathbb{R}^{\bar N}$ for $\bar x$. As before, $\bar l={\bf 0} \in  \mathbb{R}^{\bar N}$. Further, we use the upper bound vector $\bar u=(\bar u^1,\dots,\bar u^n)^T$, where $\bar u^1=(\nu,\dots,\nu, u^1)^T\in \mathbb{R}^{2p+mp}$ begins with $2p$ entries of value $\nu=\sum\limits_{i=1}^p\sum\limits_{j=1}^n |w_i^j|$. Further $\bar u^i=(0,\dots,0,u^i)^T\in \mathbb{R}^{2p+mp}$ begins with $2p$ entries $0$ for all $i>1$.  Note that in a feasible solution no slack variable ever is larger than $\nu$ and that $\nu$ has an encoding length that is polynomial in the input size of $(P_2)$.

Finally, we construct the new matrix $\bar C=(\bar C^1,\dots,\bar C^n)\in \mathbb{R}^{dp\times N}$ of $n$ building-blocks $\bar C^j= [ {\bf{0}} \quad (C')^j \dots (C')^j]\in \mathbb{R}^{dp\times 2p+mp}$, where ${\bf{0}}\in \mathbb{R}^{dp\times 2p}$ is a $0$-matrix and there again are $m$ copies of the $(C')^j $ in the blocks $\bar C^j$. Essentially, all parts that correspond to slack variables are ignored for the objective function value, i.e. $f(\bar C \bar x) = f(Cx)$.

The new system $\bar A^{(n)}\bar x=\bar b', \bar l\leq \bar x \leq \bar u$, as well as the new objective function matrix $\bar C$, extend the original system only by introducing an additional polynomial number of zeroes, ones and copies of numbers from the input. The matrices $\bar A_1, \bar A_2$ still have a constant number of rows and columns, and $\bar C$ has a constant number of rows.

It remains to discuss why the new system is equivalent to the constraints of $(P_2)$. Clearly, the '$x$-part' of $\bar x$ still has to be a $0,1$-solution. These $0,1$-entries have the same role as in the original construction, as in the $\bar A_2$-part the slack variables are only combined with zeroes.

The lower bounds on the slack variables force them to be non-negative. By the upper bounds, only the slack variables in the first block of the construction can be greater than zero. The other slack variables are only in the system to preserve the $n$-fold structure of the matrix. Note that it suffices to allow integral values for the slack variables (recall $\bar x \in \mathbb{Z}^{\bar N}$), as both the $w_i^j$ and the $b_i^{\pm}$ are integral. The rows of the $\bar A_1$-blocks then guarantee that the cluster sizes plus the slack variables add up to the given lower and upper bounds. In particular, the clusters' sizes lie in the given, bounded range.

Thus the new system corresponds to $(P_2)$ and satisfies all prererequisites of Proposition \ref{prop:polytime}, in particular $\bar A_1, \bar A_2$, and $\bar C$ do so.
\end{proof}

Finally, we prove Theorem \ref{thm:polysolve3} by extending the construction in the proof of Theorem \ref{thm:polysolve2} to allow for $s$ different features for each item. We continue with the notation introduced for the proofs of both Theorem \ref{thm:polysolve} and \ref{thm:polysolve2}.

\begin{proof}[{\bf \em Proof of Theorem $3$}]
We want a problem statement of $(P_3)$ of the form
\[\max \{ f(\tilde C\tilde x): \tilde A^{(n)}\tilde x=\tilde b', \tilde l\leq \tilde x \leq \tilde u, \tilde x\in \mathbb{Z}^{\tilde N}\}, \tag*{$(P^*)$}\]
where $\tilde A^{(n)}$ is an $n$-fold matrix derived by the standard construction for matrices $\tilde A_1$ and $\tilde A_2$.

We use $\tilde x =(\tilde x^1,\dots, \tilde x^n)^T$, where the column-blocks $\tilde x^j=((S^j)^+,(S^j)^-,x^j)^T$ all contain their own copies of the slack vectors $(S^j)^+=((S_1^j)^+,\dots,(S_p^j)^+)^T$ and $(S^j)^-=((S_1^j)^-,\dots,(S_p^j)^-)^T$. Herewith $\tilde x \in \mathbb{R}^{\tilde N}$ for $\tilde N=n(2sp+mp)$, which is polynomial in the size of the input.

For the $n$-fold construction of $\tilde A^{(n)}$, let now $$\tilde A_1 = \left[ {\begin{array}{ccc} E & {\bf 0} & A_1' \\  {\bf 0}  & -E & A_1'   \end{array} }\right]\in\mathbb{R}^{2sp\times (2sp+mp)} \text{ with } $$ $$A_1' = \left[ {\begin{array}{ccccccccc} W_1^1 & & &  & & W_1^m & & & \\  & W_2^1 & &  & \dots &  & W_2^m & & \\ & & \dots &  &  \dots &  & & \dots & \\  &  & & W_p^1  & &  &  & & W_p^m \end{array} } \right] \in \mathbb{R}^{sp \times mp},$$
where $E\in \mathbb{R}^{sp\times sp}$ is the unit-matrix and ${\bf 0}\in \mathbb{R}^{sp\times sp}$ the $0$-matrix. Further, let $$\tilde A_2 =[0\quad  0\quad  A_2]\in \mathbb{R}^{1\times (2sp+mp)}.$$
The new right-hand side vector is  $\tilde b'=(\tilde b, {\bf 1})^T\in \mathbb{R}^{2sp+n}$ with ${\bf{1}}=(1,\dots,1)^T\in \mathbb{R}^n$ and $\tilde b=(B_1^+,\dots,B_p^+,B_1^-,\dots,B_p^-)^T$ listing the upper and lower bound vectors on the cluster sizes.

Again, we use $\tilde l={\bf 0}\in \mathbb{R}^{\tilde N}$ as lower bounds. The upper bound vector $\tilde u=(\tilde u^1,\dots,\tilde u^n)^T \in  \mathbb{R}^{\tilde N}$ contains $\tilde u^1=(\tilde \nu,\dots,\tilde \nu, u^1)^T\in \mathbb{R}^{2sp+mp}$ which begins with  $2p$ vectors $\tilde \nu=\sum\limits_{i=1}^p\sum\limits_{j=1}^n |W|_i^j \in \mathbb{R}^s$ (where $|W|_i^j$ refers to the vector listing the absolutes in $W_i^j$ componentwisely), and $\bar u^i=(0,\dots,0, u^i)^T\in \mathbb{R}^{2sp+mp}$ begins with $2sp$ entries $0$ for all $i>1$. The $u^j =(u_1^j,\dots,u_m^j)^T$ are defined by setting $u_i^j=1\in\mathbb{Z}^p$ if item $j$ has weight matrix $W^i$  and  $u_i^j=0\in\mathbb{Z}^p$ otherwise.

Finally, as in the construction for Theorem \ref{thm:polysolve2}, it is necessary to ignore the parts that correspond to slack variables for the objective function value. It is possible to do so by means of the matrix $\tilde C=(\tilde C^1,\dots,\tilde C^n)\in \mathbb{R}^{dp\times \tilde N}$ that consists of $n$ building-blocks $\tilde C^i= [ {\bf{0}} \quad (C')^1 \dots (C')^1]\in \mathbb{R}^{dp\times 2sp+mp}$, where ${\bf{0}}\in \mathbb{R}^{dp\times 2sp}$ is a $0$-matrix. Then $f(\tilde C \tilde x) = f(\bar C \bar x) = f(Cx)$.

It remains to explain why this formulation represents the constraints of $(P_3)$. By definition of $\tilde l$ and $\tilde u$, the '$x$-part' of $\tilde x$ still is a $0,1$-solution. These $0,1$-entries play the same role as in the original construction by definition of the $\bar A_2$.  The slack variables are non-negative, and by the upper bounds only the slack variables in the first block of the construction can be greater than zero. The components of $\tilde \nu$ are sufficently large to not impose a restriction. Again it suffices to allow integral values for the slack variables, as both the $W^j$ and the $B_i^{\pm}$ are integral. The rows of the $\tilde A_1$-blocks then guarantee that the cluster sizes plus the slack variables add up to the given lower and upper bounds for each component of the weight vectors of the clusters, which implies that all of the clusters' total weights lie within the range given by the  $B_i^{\pm}$.

Thus the new system corresponds to $(P_3)$ and satisfies all prererequisites of Proposition \ref{prop:polytime}.
\end{proof}

\section{Polynomial-time land consolidation}

Our interest in the presented clustering framework whose favorable complexity we studied in Sections $2$ and $3$ originally came from an application in land consolidation. We now conclude our discussion with a new model and algorithm for it. Recall the description of the problem in Section $1.1$.

First, let us connect the general definitions used for model $(P_3)$ with the application. We here extend \cite{bbg-11}, which was a particularly useful model in practice. The items are the $n$ lots which have to be divided among the $p$ farmers $\pi_1,\dots,\pi_p$ (the clusters). We begin with the balancing constraints.

\subsection{Weight-balancing constraints}

The lots differ in $s$ different features, for example size, value (which is impacted by the quality of soil), attached subsidies, and so on. While some of these measures are independent of which farmer the lot is assigned to - for example the size and quality of soil are just fixed numbers - for others this may not be the case. The subsidies a farmer gets for cultivating a lot depend on several factors that differ between farmers - for example depending whether the farmer represents a small local family or a large agricultural business working in multiple regions.

For each lot, we set up a matrix $W^j\in \mathbb{R}^{s \times p}$. The column vectors $W_i^j$ list the contribution of lot $j$ to the total of a farmer  $\pi_i$ with respect to the $s$ features, if the lot is assigned to the farmer. The objective functions in the following will use the sizes of the lots of each farmer. The information on the size of lots is represented as one of the components in each column $W_i^j$ of $W^j\in \mathbb{R}^{s\times p}$. We give the size of lot $j$ the explicit name $\omega_j$.

Let $B_i\in \mathbb{R}^s$ list the summed-up total features for farmer $\pi_i$. In the redistribution process, the farmers do not accept a large deviation with respect to any of the features. For an accepted change of for example $3\%$, one obtains $B_i^-=0.97\cdot B_i$ and $B_i^+=1.03\cdot B_i$. The accepted deviations may also be defined differently for the different features.

\subsection{Objective function}

In the following, in a generalization of the model in \cite{bbg-11}, we discuss a family of objective functions for land consolidation that fit with our framework. Each lot $j$ is represented by its midpoint $z_j\in \mathbb{R}^2$. Further, the farmers $\pi_i$ specify a location $v_i\in \mathbb{R}^2$ of their farmstead. The distance of a lot and a farmstead is measured by the square of their Euclidean distance $\| v_i-z_j\|^2$.

We begin with the classical least-squares assignment for a given, single location $v_i$ of a farmstead (as a basic building block of our model). The lots are assigned to the farmers such that
\[\sum\limits_{i=1}^p\sum\limits_{j\in\pi_i} \omega_j\|v_i-z_j\|^2  \tag*{$(f_1)$}\] is minimized. The sizes $\omega_j$ of the lots are used as scaling factors for the distances $\|v_i-z_j\|^2$ to have a fair treatment of the assignment of one large lot or of many small lots.

It remains to explain why this objective function fits our framework, i.e. that it can be represented by a function $f:\mathbb{R}^{d\times p}\rightarrow \mathbb{R}$ that is passed the $p$ arguments of the form $\sum\limits_{j\in \pi_i} C_i^j $ and is convex on each of these sums \cite{or-04}:

For each item $j$, one can use $C^j_i=-\omega_j \|v_i-z_j\|^2$, so that  $\sum\limits_{j\in \pi_i} C_i^j =- \sum\limits_{j\in \pi_i}\omega_j \|v_i-z_j\|^2$. It then suffices to maximize the linear function $f$ that sums up its arguments. Formally $f(y)={\bf 1}^Ty$ with ${\bf 1} = (1,\dots,1)^T\in \mathbb{R}^{p}$. Note $d=1$.

Such a least-squares assignment favors a good assignment of lots of a farmer with larger total size over a good assignment of a smaller farmer. Thus one may want to 'normalize' the different parts of this sum. We introduce this new approach as {\bf normed least-squares assignment}.

 For this, define $C^j_i=\binom{-\omega_j \|v_i-z_j\|^2}{\omega_j} \in \mathbb{R}^2$ and use a function $f:\mathbb{R}^{2\times p}\rightarrow \mathbb{R}$ that sends $y=\left(\binom{y_{11}}{y_{12}},\dots,\binom{y_{p1}}{y_{p2}}\right)$ to $f(y)=\sum\limits_{i=1}^p \frac{y_{i1}}{y_{i2}}$. As $\sum\limits_{j\in \pi_i} C_i^j =\binom{- \sum\limits_{j\in \pi_i} \omega_j \|v_i-z_j\|^2}{\sum\limits_{j\in\pi_i} \omega_j}$, one obtains a maximization of \[- \sum\limits_{i=1}^p \frac{1}{\sum\limits_{j\in\pi_i} \omega_j} \sum\limits_{j\in\pi_i} \omega_j\|v_i-z_j\|^2,\] or equivalently, and more intuitively, a minimization of
\[\sum\limits_{i=1}^p \frac{1}{\sum\limits_{j\in\pi_i} \omega_j} \sum\limits_{j\in\pi_i} \omega_j\|v_i-z_j\|^2.\tag*{$(f_2)$}\]
By scaling by the inverse of each farmer's total assigned land, each farmer's quality of assignment contributes equally to the final objective function value.

Unfortunately, $f(y)=\sum\limits_{i=1}^p \frac{y_{i1}}{y_{i2}}$ is not convex, even when restricted to a strictly positive domain, so it is necessary to resort to a (provably good) approximation for $f_2$. For this, denote the original total size of each farmers' lots by $\kappa_i$ and the accepted lower and upper bounds on the sizes as $\kappa^{\pm}$. One can then use these apriori $\kappa_i$ to estimate the total size of lots of each farmer after the redistribution. (Compare this to the use of an approximate center of gravity in \cite{bg-12}.) This results in a minimization of \[\sum\limits_{i=1}^p \frac{1}{\kappa_i} \sum\limits_{j\in\pi_i} \omega_j\|v_i-z_j\|^2,\tag*{$(f_3)$}\]
which is a linear transform of a least-squares assignment and thus fits our framework, too. One can easily see this by using  $C^j_i=-\frac{\omega_j}{\kappa_i} \|v_i-z_j\|^2$. The approximation error of optimizing $f_3$ in place of $f_2$ is provably low.

\begin{lemma}\label{lem:approx}
Let $\pi$ be an optimal partition for $f_3$ and $\pi'$ be optimal for $f_2$. Then $\pi$ is a $(\max\limits_{i\leq p} \frac{\kappa_i}{\kappa_i^-})(\max\limits_{i\leq p}\frac{\kappa_i^+}{\kappa_i})$--approximation with respect to $f_2$.
\end{lemma}

\begin{proof}
 For a simple notation, we refer to the corresponding objective function values as $f_3(\pi)$ and $f_2(\pi')$. As $\kappa_i^- \leq \sum\limits_{j\in\pi_i} \omega_j \leq \kappa_i^+$, one has both $f_2(\pi)\leq (\max\limits_{i\leq p} \frac{\kappa_i}{\kappa_i^-}) f_3(\pi)$ and
$f_3(\pi)\leq f_3(\pi')\leq (\max\limits_{i\leq p}\frac{\kappa_i^+}{\kappa_i})f_2(\pi')$, which combines to $f_2(\pi)\leq  (\max\limits_{i\leq p} \frac{\kappa_i}{\kappa_i^-}) \cdot  (\max\limits_{i\leq p}\frac{\kappa_i^+}{\kappa_i}) \cdot f_2(\pi')$. Thus we obtain a $(\max\limits_{i\leq p}  \frac{\kappa_i}{\kappa_i^-})(\max\limits_{i\leq p}\frac{\kappa_i^+}{\kappa_i})$--approximation.
\end{proof}

For an accepted upper and lower deviation of $3\%$, $f_3$ would yield a $1.0609$--approximation error for $f_2$. Note that $f_3$ is linear and is defined with $d=1$, but uses the fact that $C_{i_1}^j$ and $C_{i_2}^j$ can differ for $i_1 \neq i_2$. In the literature there are several examples for objective functions that use a larger $d$, but that only fit our framework for $\kappa_i^-=\kappa_i^+$ for all $i\leq p$. Examples include finding a partition of minimal variance \cite{dhorw-09, or-04} and pushing apart the centers of gravity of the partition parts \cite{bbg-14, bg-10}. Both of them can then be interpreted as norm-maximization over a gravity polytope or a shaped partition polytope; recall the final remarks in Section $2$.


\subsection{Algorithm and Efficiency}\label{sec:complexity}

Theorem \ref{thm:polysolve3} implies that for a fixed number $p$ of farmers, and if there is a fixed set $\Omega$ of vectors of lot features, the above model is solvable in polynomial time for $(f_1)$ and $(f_3)$ in the size of the input, in particular $n$.  (Recall that due to the way the features of the lots are measured, this is not a particularly restrictive assumption.) In combination with Lemma \ref{lem:approx}, we also obtain a polynomial-time approximation for $(f_2)$. Let us sum up this information.

\begin{theorem}\label{thm:landcomplexity}
Suppose the number of farmers $p$ is fixed and suppose there is a fixed set $\Omega \subset \mathbb{Z}^{s\times p}$ of weight matrices such that all $W^j \in \Omega$. Then the above model for land consolidation can be solved exactly in polynomial time for objective functions $f_1$ and $f_3$. For $f_2$ a $(\max\limits_{i\leq p} \frac{\kappa_i}{\kappa_i^-})(\max\limits_{i\leq p}\frac{\kappa_i^+}{\kappa_i})$--approximation can be computed in polynomial time.
\end{theorem}

Algorithm \ref{algo:polytime} describes the method in pseudocode for objective function $(f_3)$; we use the notation from the proof of Theorem  \ref{thm:polysolve3}. First, the input data is used for a formulation $(P^*)$  of the problem in the form
\[\max \{ f(\tilde C\tilde x): \tilde A^{(n)}\tilde x=\tilde b', 0 \leq \tilde x \leq \tilde u, \tilde x\in \mathbb{Z}^{\tilde N}\}. \tag*{$(P^*)$}\]
For better readability, we have split up the description in two parts: First, one sets up a problem statement in the form $(P_3)$ (step $1$), then this formulation is transformed to the form $(P^*)$ (step $2$). In our practical implementation, we of course construct the matrices and vectors for $(P^*)$ directly. In Section $3$, we showed that this can be done in polynomial time and that $(P^*)$ is of polynomial size in the input.

The input to the problem is an actual distribution of farmland, which we represent by $y\in \{0,1\}^{p\times n}$ with $y_{ij}=1$ if lot $j$ belongs to farmer $i$ and $y_{ij}=0$ otherwise. In particular, the original distribution is feasible with respect to the given lower and upper bounds $B_i^{\pm}$. Thus there always is a feasible solution to $(P_3)$, and thus to $(P^*)$, as well. The proofs of Theorem \ref{thm:polysolve} and \ref{thm:polysolve3} explain how to construct an initial feasible solution $\tilde x_*$ for $(P^*)$ from $y$. This is done in step $3$; the original vector is `expanded' by a polynomial number of zeroes in the correct components. Analogously, in the final step $5$ of the algorithm, it then is easy to return an optimal vector of decision variables $y^*\in \{0,1\}^{p\times n}$ from the optimal solution $\tilde x^*$ for $(P^*)$.


Steps $1$ to $3$ provide a suitable formulation $(P^*)$ and an initial feasible solution, such that the problem can be solved efficiently with tools from the literature. In step $4$, one
iteratively augments this initial feasible solution to an optimal one using Graver bases methods. These methods were first shown to run in polynomial time in \cite{dhow-08}
and drastically improved in \cite{hor-13} which is the fastest algorithm to-date.

As a service to the reader, let us briefly explain the main ingredients of
this approach. For a more detailed background see the survey \cite{o-12}
and the books \cite{dhk-13, o-10}, as well as the papers \cite{dhorw-09, dhow-08, hor-13}.
The basic notion underlying these methods is the so-called {\em Graver basis}
$\mathcal{G}(\tilde A^{(n)})$ of the matrix $\tilde A^{(n)}$. Informally, this is a finite set of vectors $g$ with some special properties: in particular, if a current solution $\tilde x$ is not optimal, and $f$ is linear, there is a $g\in  \mathcal{G}(\tilde A^{(n)})$ and an $\alpha\in \mathbb{N}$ such that $\tilde x + \alpha\cdot g$ is another feasible solution and satisfies $f(\tilde x)< f(\tilde x + \alpha\cdot g)$. This gives rise to an iterative augmentation scheme, which is performed in step $4$.

Recall that $\Omega$ is fixed in our setting. In turn, this implies that the matrices $\tilde A_1$ and $\tilde A_2$ are fixed
(the proof of this fact was one of our main concerns in Section $3$). Thus the Graver basis $\mathcal{G}(\binom{\tilde A_1}{\tilde A_2})$
of the matrix $\binom{\tilde A_1}{\tilde A_2}$ is also fixed. The key theorem in the theory of $n$-fold integer programming states that in this situation 
the size of $\mathcal{G}(\tilde A^{(n)})$ is polynomial in $n$ and it can be computed in time polynomial in $n$ (see Theorem $4.4$, Chapter $4$, in \cite{o-10}).

If $f$ is arbitrary convex, one cannot avoid computing the entire Graver basis, which is polynomial in $n$ but of large degree,
and use the algorithm in \cite{dhorw-09}. However, due to $(f_3)$ being linear, it is possible to use the drastically stronger result of \cite{hor-13}.
We briefly explain how this works. To show that this scheme is efficient and polynomial, it is necessary to check that both the number
of augmentations and the computational effort for each augmentation is polynomial in the input. 

For the number of augmentations, it can be shown that if one performs so-called {\em Graver-best augmentation steps}, i.e. one always uses a best 
combination of $g$ and $\alpha$ to improve on the current feasible solution, one then obtains an upper bound $O(nL)$ on the number of necessary steps,
where $L$ is the bit size of the input, that is, a bound which is linear in $n$ and in $L$ (see the proof of  Lemma $3.10$, Chapter $4$, in \cite{o-10}).

To find the Graver-best augmentation step at each iteration, the first approach, taken in
\cite{dhow-08}, was to compute the entire Graver basis, which, as said, has size polynomial in $n$ but of large degree. However, in \cite{hor-13} it was shown that using a sophisticated 
dynamic program in each iteration, it is possible to find a Graver-best augmentation step
at that iteration {\em without} explicitly constructing the Graver basis, in time $O(n^2)$. 
This leads to a very fast and practically implementable total running time $O(n^3L)$ for
augmenting the initial feasible solution to an optimal one.

\begin{algorithm}[H]
\vspace{0.2cm}

 {\bf Input}
\begin{itemize}
\item $p$ farmers, and for each farmer $i$
\vspace*{-0.15cm}
\begin{itemize}
\item accepted lower and upper bounds $B_i^{\pm}\in \mathbb{R}^s$ with respect to $s$ features of the lots
\item the geographical location $v_i\in \mathbb{R}^2$ of a farmstead
\end{itemize}
\item $n$ lots, and for each lot $j$
\vspace*{-0.15cm}
\begin{itemize}
\item a matrix $W^j=(W_1^j,\dots,W_p^j)\in \mathbb{R}^{s\times p}$ of contributions $W_i^j\in \mathbb{R}^s$ with respect to all $s$ features if lot $j$ is assigned to farmer $i$. $\omega_j$ denotes the size of lot $j$. The $W_i^j$ come from a finite, fixed domain $\Omega$.
\item its geographical location $z_j\in \mathbb{R}^2$
\end{itemize}
\item Original land distribution represented by $y\in \{0,1\}^{p\times n}$ with $y_{ij}=1$ if lot $j$ belongs to farmer $i$ and $y_{ij}=0$ otherwise. $\kappa_i$ denotes the original total size of all lots of farmer $i$, and is derived from $y$ and the $\omega_j$.
\end{itemize}

\vspace{0.1cm}
We denote by $L$ the input size which is the total number of bits in the binary encoding of the data $(B_i^\pm, v_i, \omega_j, z_j, y)$. The matrices $W^j$ are fixed and hence do not have to be considered for $L$.

\vspace{0.1cm}
\noindent {\bf Output}
\begin{itemize}
\item Partition of the lots into $p$ clusters $\pi_1,\dots,\pi_p$ that is optimal with respect to $(f_3)$, i.e. an optimal solution for
 \[\min\sum\limits_{i=1}^p \frac{1}{\kappa_i} \sum\limits_{j\in\pi_i} \omega_j\|v_i-z_j\|^2\]
that adheres to the accepted lower and upper bounds
\end{itemize}

\noindent {\bf Algorithm}
  \begin{enumerate}
 \item Use slack vectors $S_i^{\pm}\in \mathbb{Z}^s$, $f:\mathbb{R}^p\rightarrow \mathbb{R}$ with $f(y)=(1,\dots,1)^Ty$ , and matrix $C=(C_1^j,\dots,C_p^j)\in \mathbb{R}^{1\times p}$ of utility $C^j_i=-\frac{\omega_j}{\kappa_i} \|v_i-x_j\|^2$ of lot $j$ to describe the problem in the form $(P_3)$ (see Section \ref{sec:model})
\item Transform the formulation $(P_3)$ to the form $(P^*)$ (see the proof of Theorem \ref{thm:polysolve3})
 $$\max \{ f(\tilde C\tilde x): \tilde A^{(n)}\tilde x=\tilde b', 0 \leq \tilde x \leq \tilde u, \tilde x\in \mathbb{Z}^{\tilde N}\}$$
\item Construct a feasible solution $\tilde x_*$ for $(P^*)$ from $y$ (see the proofs of Theorems \ref{thm:polysolve} and \ref{thm:polysolve3})
\item Augment the initial feasible solution $\tilde x_*$ to an optimal one $\tilde x^*$
 by iteratively using Graver-best augmentation steps using the algorithm of \cite{hor-13}, in total time $O(n^3L)$. 
\item Construct a decision variable vector $y^*\in \{0,1\}^{p\times n}$ that corresponds to the optimal $\tilde x^*$ and return $y^*$ (see the proofs of Theorems \ref{thm:polysolve} and \ref{thm:polysolve3})
\end{enumerate}
  \caption{Polynomial-Time Land Consolidation}
   \label{algo:polytime}
\end{algorithm}

\section*{Acknowledgements}
\noindent The first author gratefully acknowledges support from the Alexander-von-Humboldt Foundation. The second author was partially supported by the Dresner Chair at the Technion.


\bibliographystyle{plain}
\bibliography{literature}

\end{document}